\documentclass{amsart}      
\usepackage{amssymb,amsthm, amsmath, amsfonts} 
\usepackage{graphics}                 
\usepackage{hyperref}                 
\usepackage[all]{xy}
\usepackage{enumerate}

\newtheorem{theorem}{Theorem}[section]
\newtheorem{lemma}[theorem]{Lemma}
\newtheorem{proposition}[theorem]{Proposition}

\newtheorem{corollary}[theorem]{Corollary}

\newtheorem{example}[theorem]{Example}
\newtheorem{remark}[theorem]{Remark}
\numberwithin{equation}{section}

\DeclareMathOperator{\pd}{pd}
\DeclareMathOperator{\Hom}{Hom}
\DeclareMathOperator{\End}{End}
\DeclareMathOperator{\Aut}{Aut}
\DeclareMathOperator{\module}{-mod}
\DeclareMathOperator{\gldim}{gl.dim}
\DeclareMathOperator{\fdim}{fin.dim}
\DeclareMathOperator{\sgldim}{sgl.dim}

\DeclareMathOperator{\additive}{add}
\DeclareMathOperator{\trace}{tr}
\DeclareMathOperator{\Ext}{Ext}

\title{Homological dimensions of crossed products}
\author{Liping Li}
\address{Department of Mathematics, University of California, Riverside, CA, 92521.}
\email{lipingli@math.ucr.edu}
\thanks{The author would like to thank Prof. Avramov and Prof. Kleiner for their valuable comments. On the 2014 Maurice Auslander International Conference the author described this work. Prof. Avramov kindly helped the author figure out a reasonable definition of lengths for objects in right bounded derived categories of left Noetherian rings. Prof. Kleiner pointed out that some results in this paper can be deduced from theory of separable extensions.}
\begin{document}

\begin{abstract}
In this paper we consider several homological dimensions of crossed products $A _{\alpha} ^{\sigma} G$, where $A$ is a left Noetherian ring and $G$ is a finite group. We revisit the induction and restriction functors in derived categories, generalizing a few classical results for separable extensions. The global dimension and finitistic dimension of $A ^{\sigma} _{\alpha} G$ are classified: global dimension of $A ^{\sigma} _{\alpha} G$ is either infinity or equal to that of $A$, and finitistic dimension of $A ^{\sigma} _{\alpha} G$ coincides with that of $A$. A criterion for skew group rings to have finite global dimensions is deduced. Under the hypothesis that $A$ is a semiprimary algebra containing a complete set of primitive orthogonal idempotents closed under the action of a Sylow $p$-subgroup $S \leqslant G$, we show that $A$ and $A _{\alpha} ^{\sigma} G$ share the same homological dimensions under extra assumptions, extending the main results in \cite{L1, L2}.
\end{abstract}

\maketitle

\section{Introduction}

Let $A$ be an associative ring with identity, and let $G$ be a group. Given a map $\sigma: G \to \Aut(A)$, the group of ring automorphisms of $A$, and a map $\alpha: G \times G \to U (A)$, the set of invertible elements in $A$, by imposing some conditions on them, we can define another associative ring $A ^{\sigma} _{\alpha} G$, called the \textit{crossed product} of $A$ with $G$. It is a \textit{group graded ring} (\cite{K}). For trivial $\alpha$ or trivial $\sigma$, we get a \textit{skew group ring} $A^{\sigma} G$ or a \textit{twisted group ring} $A_{\alpha} G$ correspondingly. When both maps are trivial, the crossed product coincides with ordinary group ring $AG$.

Group graded rings and their special cases crossed products and skew group rings are widely studied by many authors from the viewpoints of ring theory and representation theory; see \cite{ARS, B, CM, M, MW, P1, P2}. For instance, their homological dimensions are studied by Yi and Aljadeff in \cite{A, AY, Z}, and several criteria for the global dimension to be finite are described; Reiten and Riedtmann show that $A$ and its skew group algebra $A ^{\sigma} G$ share many properties for finite dimensional algebras $A$ and finite groups $G$ when the order of $G$ is invertible in $A$ (\cite{RR}); in \cite{L1, L2}, the author proves that these properties are still shared by $A$ and $A ^{\sigma} G$ for arbitrary finite groups $G$ under a much weaker assumption.

In this paper we mainly consider serval homological dimensions of crossed products such as global dimensions, finitistic dimensions, and strong global dimensions (defined in next section), for which we give a uniform definition by considering lengths of objects in right bounded derived categories. Based on techniques and known results introduced in \cite{A, AY, L1, L2}, we attack this problem from the viewpoint of representation theory. As the first step, we lift the classical induction and restriction functors to homotopy categories of complexes with finitely generated projective components, and show that some natural maps between modules give rise to chain maps. Consequently, $A ^{\sigma} _{\alpha} G$ and $A$ share the same homological dimensions. This phenomenon actually happens for quitely many \textit{separable extensions} (defined in Section 3). That is,

\begin{theorem}
Let $R$ be a left Noetherian ring and let $S$ be a subring which is also left Noetherian. Suppose that $1_S = 1_R$, $_SR$ is a finitely generated projective $S$-module, and $R_S$ is a flat right $S$-module. If $R$ is a separable extension over $S$, then $R$ and $S$ have the same global dimension, finitistic dimension, and strong global dimension.
\end{theorem}

If $A$ is further commutative, by a result of Aljadeff (\cite{A}), a skew group ring $A ^{\sigma} G$ has finite global dimension if and only if so does $A$ and the trivial $A ^{\sigma} G$-module $A$ is projective (the second condition is true if and only if $A ^{\sigma} G$ is a separable extension over $A$ by \cite{NVV}). We classify global dimensions and finitistic dimensions of crossed products for arbitrary left Noetherian rings, and extend his result to noncommutative rings.

\begin{theorem}
Let $A ^{\sigma} _{\alpha} G$ be a crossed product, where $A$ is left Noetherian and $G$ is a finite group.
\begin{enumerate}
\item For every $M \in A ^{\sigma} _{\alpha} G \module$, its projective dimension $\pd _{A ^{\sigma} _{\alpha} G} M$ is either infinity or equal to $\pd_A M$ (when $\pd_A M < \infty$). Correspondingly, the global dimension $\gldim A ^{\sigma} _{\alpha} G$ is either infinity or equal to $\gldim A$ (when $\gldim A < \infty$), and the finitistic dimension $\fdim A ^{\sigma} _{\alpha} G = \fdim A$.
\item A skew group ring $A ^{\sigma} G$ has finite global dimension if and only $\gldim A < \infty$ and the trivial representation $A$ is a projective $A ^{\sigma} G$-module. Moreover, if the trivial representation is projective, $A ^{\sigma} H$ and $A$ have the same global dimension for every subgroup $H \leqslant G$.
\end{enumerate}
\end{theorem}

Unfortunately, in practice it is not easy to check whether the trivial representation is projective. However, many algebras (in particular finite dimensional algebras) are defined by using certain combinatorial structures such as quivers, posets, categories, etc. In this situation the actions of groups on the sets of vertices or the sets of objects play a central role. We then focus on a special case that $A$ is a semiprimary left Noetherian algebra over an algebraically closed field $k$, and suppose that there are a Sylow $p$-subgroup $S \leqslant G$ and a complete set $E = \{ e_i \} _{i \in [n]}$ of primitive orthogonal idempotents in $A$ closed under the action of $S$. Denote by $C(A)$ and $A^S$ the center of $A$ and the fixed algebra respectively. The following conclusion gives us a feasible criterion for the global dimension of $A ^{\sigma} _{\alpha} G$ to be finite, generalizing a main result in \cite{L1, L2} for skew group algebras.

\begin{theorem}
Let $A$ be a finite dimensional algebra, and let $G$, $S$, and $E$ be as above. Suppose that there is a domain $D \in C(A) \cap A^S$ containing $\alpha (x, y)$ and the $|S|$-th root of $h_x = \prod _{y \in S} \alpha (x, y)$ for every $x, y \in S$. Then $A _{\alpha} ^{\sigma} G $ has finite global dimension if and only if so does $A$ and the action of $S$ on $E$ is free. Furthermore, if $S$ acts freely on $E$, then $A _{\alpha} ^{\sigma} G$ and $A$ have the same global dimension.
\end{theorem}

The paper is organized as follows. For convenience of the reader, in the next section we describe some preliminary results on crossed products, and introduce a uniform definition for the homological dimensions we study in this paper. In Section 3 we lift the classical induction and restriction functors between module categories to functors between homotopy categories, and revisit many classical results in the derived categories. Theorem 1.1 is proved there. In Section 4 we focus on global dimensions and prove Theorem 1.2. The last section is devoted to the special case mentioned before. Using the strong no loop conjecture recently established in \cite{ILP}, we prove one direction of The last theorem. The other direction follows from normalization of parameters.

Here are some notations and conventions. All modules we consider in this paper are finitely generated left modules. For a ring $R$, $R\module$ is the category of finitely generated $R$-modules. By $\gldim R$, $\fdim R$, and $\sgldim R$ we mean the global dimension, finitistic dimension, and strong global dimension respectively. For $M \in R\module$, $\pd_R M$ is its projective dimension. Composition of maps and morphisms is from right to left.

\section{preliminaries}

Firstly recall the construction of crossed products. Fix an associative ring $A$ with identity, a group $G$, two maps
\begin{equation*}
\sigma: G \to \Aut (A), \quad x \mapsto \sigma_x,
\end{equation*}
and
\begin{equation*}
\alpha: G \times G \to U(A), \quad (x, y) \mapsto \alpha (x, y).
\end{equation*}
Following Marcus in \cite{M}, the pair $(\sigma, \alpha)$ is called a \textit{parameter set} of $G$ in $A$ if the following conditions are satisfied for all $x, y, z \in G$:
\begin{enumerate}
\item $\sigma_x \sigma_y = \iota _{\alpha (x, y)} \sigma_{xy}$, where $\iota _{\alpha (x, y)}$ is the inner automorphism induced by $\alpha (x, y) \in U(A)$;
\item $\alpha (x, y) \alpha (xy, z) = \sigma_x (\alpha (y, z)) \alpha (x, yz)$.
\end{enumerate}

The crossed product $A ^{\sigma} _{\alpha} G = \bigoplus _{x \in G} A \sigma_x$, where $A\sigma_x$ is a free $A$-module of rank 1 with basis $\sigma_x$. The multiplication map $\ast$ in $A^{\sigma} _{\alpha} G$ is determined by the following formula:
\begin{equation*}
(a\sigma_x) \ast (b \sigma_y) = a \sigma_x (b) \alpha (x, y) \sigma_{xy}.
\end{equation*}
The above two conditions imposed on the pair $(\sigma, \alpha)$ are equivalent to the associativity of $A ^{\sigma} _{\alpha} G$. The restricted multiplication in $A$ is denote by $\cdot$.

Two crossed products $A ^{\sigma} _{\alpha} G$ and $A ^{\sigma'} _{\alpha'} G$ are \textit{equivalent} if there is a $G$-graded algebra isomorphism $\varphi: A ^{\sigma} _{\alpha} G \to A ^{\sigma'} _{\alpha'} G$ such that the restricted automorphism $\varphi_A$ on $A$ is the identity map. Correspondingly, two parameter sets $(\sigma, \alpha)$ and $(\sigma', \alpha')$ are \textit{equivalent} if there exist elements $u_x \in U(A)$ satisfying $\sigma'_x = \sigma_x \iota_{u_x}$ and $\alpha'(x, y) = u_x \sigma_x (u_y) \alpha (x, y) u_{xy}^{-1}$. Since we only consider homological dimensions of crossed products in this paper, equivalent crossed products may be identified.

The following proposition is a direct consequence of Theorem 1.3.7 in \cite{M}.

\begin{proposition}
Two crossed products $A ^{\sigma} _{\alpha} G$ and $A ^{\sigma'} _{\alpha'} G$ are equivalent if and only if the parameter sets $(\sigma, \alpha)$ and $(\sigma', \alpha')$ are equivalent.
\end{proposition}

Given a set of invertible elements $\{ u_x \in U(A) \} _{x \in G}$, we can define a new basis $\{ \sigma'_x = u_x \sigma_x \} _{x \in G}$ for $A ^{\sigma} _{\alpha} G$. Under the new basis we get a crossed product $A ^{\sigma'} _{\alpha'} G$ which is equivalent to $A ^{\sigma} _{\alpha} G$, where $\sigma'_x = \sigma_x \iota _{u_x}$ and $\alpha' (x, y) = u_x \sigma_x (u_y) \alpha (x, y) u_{xy} ^{-1}$. Therefore, by proper basis change we can make $1_A \sigma_{1_G}$ the identity of $A ^{\sigma} _{\alpha} G$ and denote it by 1. Moreover, it is straightforward to check that $\alpha (1_G, 1_G) = \alpha (x, 1_G) = \alpha (1_G, y) = 1_A$ for all $x, y \in G$. In this paper we always assume that our crossed products satisfy these conventions.

Now we define the \textit{trivial representation} of $A ^{\sigma} _{\alpha} G$. Let $\mathfrak{I}$ be the left ideal generated by elements in $\{ \sigma_x -1 \mid 1_G \neq x \in G \}$. In general $\mathfrak{I}$ is only a left ideal, but not a right ideal. For skew group rings, $\mathfrak{I}$ consists of all elements of the form $\sum _{x \in G} a_x \sigma_x$ such that $a_x \in A$ and $\sum_{x \in G} a_x = 0$. Consequently, the quotient module $A / \mathfrak{I} \cong A$ by identifying $\sum_{x \in G} \sigma_x$ with 1, and the module action is determined by $\sigma_x \cdot a = \sigma_x (a)$ for $x \in G$ and $a \in A$. However, for crossed products, the existence of nontrivial $\alpha$ makes the structure of $\mathfrak{I}$ much more complicated. For example, since $\sigma_y - 1 \in \mathfrak{I}$, we get $\sigma_x \ast (\sigma_y -1) = \alpha(x, y) \sigma_{xy} - \sigma_x \in \mathfrak{I}$. But then $\alpha (x, y) \sigma_{xy} -1 \in \mathfrak{I}$. Since $\sigma_{xy} - 1 \in \mathfrak{I}$ as well, we deduce that $\alpha (x, y) - 1 \in \mathfrak{I}$ for every $x, y \in G$. We define the trivial representation to be the quotient $A ^{\sigma} _{\alpha} G / \mathfrak{I}$, which is clearly an $A ^{\sigma} _{\alpha} G$-module.

The following proposition gives a reason for the name of $A ^{\sigma} _{\alpha} G / \mathfrak{I}$. Recall for $M \in A ^{\sigma} _{\alpha} G \module$, $M^G$ is defined to be the set of all elements $v \in M$ such that $\sigma_x v = v$ for every $x \in G$. The \textit{fixed algebra} $A^G$ is defined to be the set of elements $a \in A$ such that $\sigma_x (a) = a$ for every $x \in G$. It is clear that $M^G$ is an $A^G$-module. Moreover, the map sending $M$ to $M^G$ is functorial since for $f \in \Hom _{A ^{\sigma} _{\alpha} G} (M, N)$, the restricted map sends $M^G$ into $N^G$. Denote this functor by $-^G$.

\begin{proposition}
Let $M, N \in A ^{\sigma} _{\alpha} G \module$. Then:
\begin{enumerate}
\item For an arbitrary $v \in M$, $v$ is contained in $M^G$ if and only if $\mathfrak{I} v = 0$.
\item There is a natural isomorphism $-^G \cong \Hom _{A ^{\sigma} _{\alpha} G} (A ^{\sigma} _{\alpha} G / \mathfrak{I}, -)$.
\item $\Hom _{A ^{\sigma} _{\alpha} G} (M, N) \cong \Hom_A (M, N) ^G$.
\end{enumerate}
\end{proposition}

\begin{proof}
Note that $v \in M^G$ if and only if $\sigma_x v - v = 0$ for every $x \in G$, if and only if $(\sigma_x - 1) v = 0$ for all $x \in G$. But $\mathfrak{I}$ is generated by these elements as a left $A ^{\sigma} _{\alpha} G$ ideal. The first statement follows.

To show the second one, we define two maps:
\begin{equation*}
\varphi: \Hom _{A ^{\sigma} _{\alpha} G} (A ^{\sigma} _{\alpha} G / \mathfrak{I}, M) \to M^G, \quad f \to f(\bar{1});
\end{equation*}
and
\begin{equation*}
\psi: M^G \to \Hom _{A ^{\sigma} _{\alpha} G} (A ^{\sigma} _{\alpha} G / \mathfrak{I}, M) , \quad v \to f_v
\end{equation*}
such that  $f_v (\bar{1}) = v$. Because every $\sigma_x$ fixes $\bar{1}$, these two maps are well defined. They are inverse to each other, and functorial, and hence give a natural isomorphism between $-^G$ and $\Hom _{A ^{\sigma} _{\alpha} G } (A ^{\sigma} _{\alpha} G / \mathfrak{I}, -)$ (in the framework of $A^G$-modules).

The last statement can be deduced from part (c) of Theorem 1.4.8 in \cite{M}.
\end{proof}

In the rest of this section we give a uniform definition for the homological dimensions we consider in this paper. The reader is suggested to refer to \cite{AF} and \cite{HS} for definitions of homological dimensions on complexes. For an arbitrary associative left Noetherian ring $R$ with identity we let $\additive (R)$ be the additive category of finitely generated projective $R$-modules. Denote by $C^-(_RP)$ (resp., $C^b(_RP)$) the category of right bounded complexes (resp., bounded complexes) whose terms lie in $\additive (R)$. Because $R$ is left Noetherian, $R\module$ is an abelian category. In particular, every finitely generated $R$-module has a projective resolution contained in $C^- (_RP)$. Let $K^-(_RP)$ and $K^b(_RP)$ be their homotopy categories. Objects in $K^b (_RP)$ are called \textit{perfect complexes}. Note that the right bounded derived category $D^- (R)$ of finite generated $R$-modules is equivalent to $K^- (_RP)$ as triangulated categories, and the bounded derived category $D^b (R)$ is equivalent to the full subcategory of $K^- (_RP)$ consisting of objects with bounded homologies. Thus we can identify these categories.

Given $P^{\bullet} \in K^- (_RP)$, we define $s (P^{\bullet}) = \sup \{ i \in \mathbb{Z} \mid P^i \neq 0 \}$. Similarly, $i (P^{\bullet})$ is defined to be $\inf \{ i \in \mathbb{Z} \mid P^i \neq 0\}$. The \textit{amplitude} $a (P^{\bullet})$ equals $s (P^{\bullet}) - i (P^{\bullet})$. We then define the \textit{length} $l (P^{\bullet})$ to be $\inf \{a(Q^{\bullet}) \mid Q^{\bullet} \text{ is quasi-isomorphic to } P^{\bullet} \}$. The reader readily see that for $M \in R \module$ (view it as a stalk complex in $D^b(R)$ concentrated in degree 0 and identify it with its projective resolutions), $l (M)$ is nothing but the projective dimension of $M$. Therefore,
\begin{align*}
& \gldim R = \sup \{ l(P^{\bullet}) \mid P^{\bullet} \in K^- (_RP) \text{ and } H^i(P^{\bullet}) \neq 0 \text{ for at most one } i \in \mathbb{Z} \};\\
& \fdim R = \sup \{ l(P^{\bullet}) \mid P^{\bullet} \in K^b (_RP) \text{ and } H^i(P^{\bullet}) \neq 0  \text{ for at most one } i \in \mathbb{Z} \};\\
& \sgldim R = \sup \{ l(P^{\bullet}) \mid P^{\bullet} \in K^b (_RP) \text{ is indecomposable} \}.
\end{align*}

The concept strong global dimension was introduced by Ringel for finite dimensional algebras and he conjectured that a finite dimensional algebra has finite strong global dimension if and only if it is \textit{piecewise hereditary}; that is, its bounded derived module category is equivalent to the bounded derived category of a hereditary abelian category as triangulated categories. This conjecture was proved recently by Happel and Zacharia. For more details, see \cite{HZ1, HZ2}. Note that we always have $\sgldim R \geqslant \gldim R \geqslant \fdim R$. The second inequality is obvious, while the first one can be observed by taking truncations of projective resolutions of finitely generated $R$-modules.

Several interesting open questions are related to these homological dimensions. The famous finitistic dimension conjecture asserts that the finitistic dimension of an artinian algebra is always finite. For a finite dimensional hereditary algebra, its global dimension and strong global dimension coincide. But in general we do not know for what algebras this equality holds.

\section{Induction and Restriction}

In this section we consider two classical functors: induction and restriction. Many techniques and results stated in this section are well known for module categories, and our goal is to revisit them in derived categories. This slight generalization is essential for studying strong global dimension since it cannot be defined in module categories as global dimension or finitistic dimension.

Let $R$ be a left Noetherian ring and $S$ be a left Noetherian subring. We also suppose that $1_R = 1_S$ and $_SR$ is a finitely generated $S$-module. For $M \in S \module$, the \textit{induced module} is defined to be $R \otimes_S M$, which is finitely generated. For $N \in R \module$, the \textit{restricted module} is $_SN$, which is finitely generated as well since $_SR$ is finitely generated. In this way we get a pair of adjoint functors $\uparrow _S^R$ and $\downarrow _S^R$. That is, there is a natural isomorphism $\Hom_R (M \uparrow _S^R, N) \cong \Hom_S (M, N \downarrow _S^R)$.

According to \cite{HS}, $R$ is a \textit{separable extension} over $S$ if the multiplication epimorphism $R \otimes_S R \to R$ by sending $a \otimes b$ to $ab$ is split. In other words, there is a certain $\sum _{i = 1}^n  a_i \otimes b_i \in R \otimes_S R$ such that $\sum _{i = 1}^n  a_i \otimes b_i r = \sum _{i = 1}^n  r a_i \otimes b_i$ for every $r \in R$ and $\sum _{i=1}^n a_ib_i = 1$.

The following proposition is well known.

\begin{proposition}
Let $R$ and $S$ be as above and suppose that $_SR$ is a finitely generated $S$-module.
\begin{enumerate}
\item If $_SR_S = S \oplus B$, for every $M \in S \module$, $M$ is isomorphic to a direct summand of $M \uparrow _S^R \downarrow _S^R$.
\item If $R$ is a separable extension over $S$, then $N$ is isomorphic to a direct summand of $N \downarrow _S^R \uparrow _S^R$ for every $N \in R \module$.
\end{enumerate}
\end{proposition}

\begin{proof}
Since $_SR_S = S \oplus B$, there is a split surjection $_SR_S \to S$, denoted by $\pi$. This gives rise to a split surjection $\pi_M: R \otimes_S M \to M$ by sending $r \otimes v \to \pi(r) v$ for $r \in R$ and $v \in M$. Its right inverse $\delta_M$ sends $v$ to $1 \otimes v$. That is, $\pi_M \circ \delta_M$ is the identity map on $M$. The first statement is proved.

Since $R$ is a separable extension over $S$, there is a certain $\sum _{i = 1}^n a_i \otimes b_i \in R \otimes_S R$ such that $\sum _{i = 1}^n  a_i \otimes b_i r = \sum _{i = 1}^n  r a_i \otimes b_i$ for every $r \in R$ and $\sum _{i=1}^n a_ib_i = 1$. For $N \in R \module$, define $\psi_N: R \otimes_S N \to N$ by sending $r \otimes v$ to $rv$ for $r \in R$ and $v \in N$, and define $\varphi_N: N \to R \otimes_S N$ by mapping $v$ to $\sum _{i = 1}^n a_i \otimes b_i v$. The first condition on separable extensions implies that $\varphi_N$ is an $R$-module homomorphism, and the second condition tells us that $\psi_N \circ \varphi_N$ is the identity map on $N$. The second statement follows immediately.
\end{proof}

A crucial observation is that the maps defined in the above proof lift to chain maps.

\begin{proposition}
Let $R$ and $S$ be as above and suppose that $_SR$ is a finitely generated projective $S$-module and $R_S$ is a finitely generated flat $S$-module.
\begin{enumerate}
\item The induction and restriction functors induce functors between $K^- (_SP)$ (resp., $K^b (_SP)$) and $K^- (_RP)$ (resp., $K^b (_RP)$), which are still denoted by $\uparrow _S^R$ and $\downarrow _S^R$.

\item If $_SR_S = S \oplus B$, for every $Q^{\bullet} \in K^- (_SP)$ (resp., $Q^{\bullet} \in K^b (_SP)$), $Q^{\bullet}$ is isomorphic to a direct summand of $Q^{\bullet} \uparrow_S^R \downarrow _S^R$.

\item If $R$ is a separable extension over $S$, then for every $P^{\bullet} \in K^- (_RP)$ (resp., $P^{\bullet} \in K^b (_RP)$), $P^{\bullet}$ is isomorphic to a direct summand of $P^{\bullet} \downarrow _S^R \uparrow _S^R$.
\end{enumerate}
\end{proposition}

\begin{proof}
We only give a proof for bounded homotopy categories as it works for right bounded homotopy categories as well. Since $1_S = 1_R$, the restriction functor is exact. Moreover, since $_SR$ is a finitely generated projective $S$-module, $\downarrow_S^R$ preserves projective modules, too. Applying it termwise to a bounded chain complex of projective $R$-modules, we get a bounded chain complex of projective $S$-modules. Moreover, it sends morphisms (homotopy classes of chain maps) in $K^b (_RP)$ to morphisms in $K^b (_SP)$. In this way we get a functor $\downarrow _S^R: K^b (_RP) \to K^b (_SP)$.

Similarly, the induction functor preserves projective modules. Moreover, it is exact since $R_S$ is flat. Given $Q^{\bullet} \in K^b (_SP)$ with differentials $(d^i) _{i \in \mathbb{Z}}$, applying the induction functor termwise we get an object $P^{\bullet} \uparrow_S^R \in K^b (_RP)$ whose $i$-th term is $R \otimes _S P^i$ and $i$-th differential map is $1 \otimes d^i$. This is a chain complex of projective $R$-modules. Moreover, given a morphism $(f^i) _{i \in \mathbb{Z}}: X^{\bullet} \to Y^{\bullet}$ in $K^b (_SP)$, we can define a corresponding morphism $(1 \otimes f^i) _{i \in \mathbb{Z}}$. These constructions are functorial. In this way we lift the induction functor from $S \module$ to $K^b (_SP)$. The first statement is proved.

To prove statements (2) and (3), it suffices to observe that under the given assumptions the homomorphisms $\delta, \pi, \varphi, \psi$ constructed in the proof of the previous proposition commute with differential maps, and hence give rise to chain maps. Explicitly, given $Q^{\bullet} \in K^b (_SP)$, we have
\begin{equation*}
\xymatrix{ Q^{\bullet} \ar[rr] ^-{\delta^{\bullet} = (\delta^i) _{i \in \mathbb{Z}}} & & Q^{\bullet} \uparrow _S^R \downarrow _S^R \ar[rr] ^-{\pi^{\bullet} = (\pi^i) _{i \in \mathbb{Z}}} & & Q^{\bullet}}
\end{equation*}
such that $\pi^{\bullet} \circ \delta ^{\bullet}$ is the identity map. Similarly, given $P^{\bullet} \in K^b (_RP)$, we have
\begin{equation*}
\xymatrix{ P^{\bullet} \ar[rr] ^-{\varphi^{\bullet} = (\varphi^i) _{i \in \mathbb{Z}}} & & P^{\bullet} \downarrow _S^R \uparrow _S^R \ar[rr] ^-{\psi^{\bullet} = (\psi^i) _{i \in \mathbb{Z}}} & & P^{\bullet}}
\end{equation*}
such that $\psi^{\bullet} \circ \varphi ^{\bullet}$ is the identity map. The conclusion follows.
\end{proof}

This proposition immediately implies Theorem 1.1.

\begin{theorem}
Let $R$ and $S$ be as above, and suppose that $_SR$ is a finitely generated projective $S$-module and $R_S$ is flat. If $R$ is a separable extension over $R$, then $R$ and $S$ have the same global dimension, finitistic dimension, and strong global dimension.
\end{theorem}

\begin{proof}
We prove the conclusion for global dimensions, and the same technique applies to other homological dimensions with very small modifications. Take an object $Q^{\bullet} \in K^- (_SP)$ with $H^i (Q^{\bullet}) \neq 0$ for at most one $i \in \mathbb{Z}$. Using quasi-isomorphisms, we can assume that the amplitude $a (Q^{\bullet})$ equals the length $l(Q^{\bullet})$.

By the above proposition, $Q^{\bullet}$ is isomorphic to a direct summand of $Q^{\bullet} \uparrow _S^R \downarrow _S^R$. We claim that $a (Q^{\bullet} \uparrow _S^R) = l (Q^{\bullet} \uparrow _S^R)$. If this is not true, then there is some $V ^{\bullet} \in K^- (_RP)$ quasi-isomorphic to $Q^{\bullet} \uparrow _S^R$ with
\begin{equation*}
a (V^{\bullet}) = l (V^{\bullet}) < a(Q^{\bullet} \uparrow _S^R) = a (Q^{\bullet}) = l (Q^{\bullet}).
\end{equation*}
Applying the restriction functor to $V^{\bullet}$, we conclude that $V^{\bullet} \downarrow _S^R$ is quasi-isomorphic to $Q^{\bullet} \uparrow _S^R \downarrow _S^R$. Consequently, $Q^{\bullet}$ is quasi-isomorphic to a direct summand of $V^{\bullet} \downarrow _S^R$, so we have:
\begin{equation*}
l (Q^{\bullet}) \leqslant a(V^{\bullet} \downarrow _S^R) = a (V^{\bullet}) = l (V^{\bullet}).
\end{equation*}
But these two inequalities contradict each other, so $a (Q^{\bullet} \uparrow _S^R) = l (Q^{\bullet} \uparrow _S^R)$ as claimed. In particular,
\begin{equation*}
l (Q^{\bullet}) = a (Q^{\bullet}) = a(Q^{\bullet} \uparrow _S^R) = l(Q^{\bullet} \uparrow _R^S) \leqslant \gldim R,
\end{equation*}
and hence $\gldim S \leqslant \gldim R$.

If $R$ is a separable extension over $S$, we can apply a similar reasoning (using (3) of the previous proposition) to every indecomposable object $P^{\bullet} \in K^- (_RP)$ with $H^i (P^{\bullet}) \neq 0$ for at most one $i \in \mathbb{Z}$ and deduce that $l(P^{\bullet}) \leqslant \gldim S$, and hence $\gldim R \leqslant \gldim S$. This forces $\gldim R = \gldim S$.
\end{proof}

We give a remark.

\begin{remark} \normalfont
The above proof implicitly implies that $\pd _R V \geqslant \pd _S V$ for arbitrary $V \in R \module$ since a projective resolution of $_RV$ restricted to $S$ give a projective resolution of $_SV$. Moreover, it is also true that $\gldim S \leqslant \gldim R$, $\fdim S \leqslant \fdim R$, and $\sgldim S \leqslant \sgldim R$. Equalities hold if $R$ is a separable extension over $S$.
\end{remark}

Now we apply the previous results to investigate homological dimensions of crossed products. Similar techniques have been used to explore other properties of crossed product in \cite{L}. Suppose that $G$ is a finite group and let $H \leqslant G$ be a subgroup. Then $\sigma: G \to \Aut (A)$ and $\alpha: G \times G \to U(A)$ restrict to $H$ and $H \times H$ respectively. Denote these restricted maps by $\sigma$ and $\alpha$ again. They define a crossed product $A _{\alpha} ^{\sigma} H$, which is a subring of $A _{\alpha} ^{\sigma} G$. It is well known that both $_{A _{\alpha} ^{\sigma} H} A _{\alpha} ^{\sigma} G$ and $A _{\alpha} ^{\sigma} G _{A _{\alpha} ^{\sigma} H}$ are finitely generated projective modules. Moreover, $_{A _{\alpha} ^{\sigma} H} A _{\alpha} ^{\sigma} G _{A _{\alpha} ^{\sigma} H} \cong A _{\alpha} ^{\sigma} H \oplus B$. For more details, see \cite{M}.

We want to show that $A _{\alpha} ^{\sigma} G$ is a separable extension over $A _{\alpha} ^{\sigma} H$ when the index $|G : H|$ is invertible in $A$. One approach is to consider the element
\begin{equation*}
\zeta = \frac{1} {|G:H|} \sum _{x \in G/H} \alpha (x, x^-) \sigma_x \otimes \sigma_{x^-} \in A _{\alpha} ^{\sigma} G \otimes _{A _{\alpha} ^{\sigma} H} A _{\alpha} ^{\sigma} G
\end{equation*}
and show that it satisfies the two conditions of separable extensions, where $x^-$ is the inverse of $x$. This approach is implicitly used in Proposition 3.3 (page 79) and Theorem 3.4 (page 81) of \cite{K}. For convenience of the reader, we give a detailed proof. We need the following technical lemma.

\begin{lemma}
For $x, y \in G$, let $x^- = x^{-1}$ and $y^- = y^{-1}$. Then:
\begin{equation}
\alpha (x, x^-)^{-1} \cdot \sigma_x (\alpha (x^-, y^-)) \cdot \alpha (y^-, yx)^{-1} = \sigma_{y^-} (\alpha (yx, x^-y^-)) ^{-1},
\end{equation}
and
\begin{equation}
\alpha (xy, y^-x^-) ^{-1} \cdot \sigma_{xy} (\alpha (y^-, x^-))^{-1} \cdot \alpha (xy, y^-) = \alpha (x, x^-)^{-1}.
\end{equation}
\end{lemma}

\begin{proof}
By considering the product $(\sigma_x \ast \sigma_{x^-}) \ast \sigma_{y^-} = \sigma_x \ast (\sigma_{x^-} \ast \sigma_{y^-})$, we get
\begin{equation*}
\alpha(x, x^-) = \sigma_x (\alpha (x^-, y^-)) \cdot \alpha (x, x^-y^-),
\end{equation*}
and hence
\begin{equation*}
\alpha(x, x^-) ^{-1} = \alpha (x, x^-y^-) ^{-1} \cdot \sigma_x (\alpha (x^-, y^-)) ^{-1}.
\end{equation*}
Applying this identity to the left side of (3.1), it suffices to show
\begin{equation*}
\alpha (x, x^-y^-) ^{-1} \cdot \alpha (y^-, yx)^{-1} = \sigma_{y^-} (\alpha (yx, x^-y^-)) ^{-1},
\end{equation*}
which is equivalent to
\begin{equation}
\alpha (y^-, yx) \cdot \alpha (x, x^-y^-) = \sigma_{y^-} (\alpha (yx, x^-y^-)).
\end{equation}

Now consider
\begin{equation*}
\sigma_{y^-} \ast ((\sigma_y \ast \sigma_x) \ast (\sigma_{x^-} \ast \sigma_{y^-})) = (\sigma_{y^-} \ast (\sigma_y \ast \sigma_x)) \ast (\sigma_{x^-} \ast \sigma_{y^-}).
\end{equation*}
On the left side we get
\begin{equation*}
\sigma_{y^-} (\alpha(y, x)) \cdot \alpha (y^-, yx) \cdot \sigma_x (\alpha (x^-, y^-)) \cdot \alpha (y^-, yx)^{-1} \cdot \sigma_{y^-} (\alpha (yx, y^-x^-)) \sigma_{y^-}.
\end{equation*}
On the right side, we have
\begin{equation*}
\sigma_{y^-} (\alpha(y, x)) \cdot \alpha (y^-, yx) \cdot \sigma_x (\alpha (x^-, y^-)) \cdot \alpha (x, x^-y^-) \sigma_{y^-}.
\end{equation*}
Therefore,
\begin{equation*}
\alpha (x, x^-y^-) = \alpha (y^-, yx) ^{-1} \cdot \sigma_{y^-} (\alpha (yx, y^-x^-)).
\end{equation*}
Multiplied by $\alpha (y^-, yx)$ on both sides, this identity turns out to be (3.3) as claimed.

The identity (3.2) is equivalent to
\begin{equation}
\sigma _{xy} (\alpha (y^-, x^-)) \cdot \alpha (xy, y^-x^-) = \alpha (xy, y^-) \cdot \alpha (x, x^-).
\end{equation}
To prove this, we consider $(\sigma_x \ast \sigma_y) \ast (\sigma_{y^-} \ast \sigma_{x^-}) = ((\sigma_x \ast \sigma_y) \ast \sigma_{y^-}) \ast \sigma_{x^-}$.
\end{proof}

\begin{proposition}
Let $G$ be a finite group. If $|G : H|$ is invertible in $A$, then $A _{\alpha} ^{\sigma} G$ is a separable extension over $A _{\alpha} ^{\sigma} H$.
\end{proposition}

\begin{proof}
Let
\begin{equation*}
\zeta = \frac{1} {|G:H|} \sum _{x \in G/H} \alpha (x, x^-) \sigma_x \otimes \sigma_{x^-} \in A _{\alpha} ^{\sigma} G \otimes _{A _{\alpha} ^{\sigma} H} A _{\alpha} ^{\sigma} G.
\end{equation*}
We claim that this is well defined. That is, it is independent of the choice of representatives in cosets. Indeeds, for a fixed $x \in G$, take another representative $xy$ with $y \in H$. Then
\begin{align*}
& \alpha (xy, y^-x^-)^{-1} \sigma_{xy} \otimes \sigma_{y^-x^-} \\
& = \alpha (xy, y^-x^-)^{-1} \sigma_{xy} \otimes (\alpha(y^-, x^-)^{-1} \sigma_{y^-} \ast \sigma_{x^-}) \\
& = \alpha (xy, y^-x^-)^{-1} (\sigma_{xy} \ast \alpha(y^-, x^-)^{-1} \sigma_{y^-}) \otimes \sigma_{x^-} \\
& = \alpha (xy, y^-x^-)^{-1} \cdot \sigma_{xy} (\alpha(y^-, x^-))^{-1} \cdot \alpha (xy, y^-) \sigma_x \otimes \sigma_{x^-} \\
& = \alpha (x, x^-)^{-1} \sigma_x \otimes \sigma_{x^-} ,
\end{align*}
where the last identity comes from (3.2). Therefore, $\zeta$ is well defined.

For $y^{-1} \in G$, we check:
\begin{align*}
& |G:H| \zeta \sigma_{y^-} = \sum _{x \in G/H} \alpha (x, x^-)^{-1} \sigma_x \otimes \sigma_{x^-} \sigma_{y^-} \\
& = \sum _{x \in G/H} (\alpha (x, x^-)^{-1} \sigma_x) \ast \alpha(x^-, y^-) \otimes \sigma_{x^-y^-} \\
& = \sum _{x \in G/H} \alpha (x, x^-)^{-1} \cdot \sigma_x (\alpha(x^-, y^-)) \sigma_x \otimes \sigma_{x^-y^-} \\
& = \sum _{x \in G/H} \alpha (x, x^-)^{-1} \cdot \sigma_x (\alpha(x^-, y^-)) \cdot \alpha (y^-, yx)^{-1} (\sigma_{y^-} \sigma_{yx}) \otimes \sigma_{x^-y^-} \\
& = \sum _{x \in G/H} \sigma_{y^-} (\alpha (yx, x^-y^-))^{-1} (\sigma_{y^-} \sigma_{yx}) \otimes \sigma_{x^-y^-} \quad \text{ by (3.1)} \\
& = \sum _{x \in G/H} \sigma_{y^-} \big{(} \alpha (yx, x^-y^-)^{-1} \sigma_{yx} \otimes \sigma_{x^-y^-} \big{)}\\
& = |G:H| \sigma_{y^-} \zeta.
\end{align*}
That is, $\zeta \sigma_{y^-} = \sigma_{y^-} \zeta$. Similarly, we check that $\zeta a = a \zeta$ for $a \in A$. Therefore, $\zeta$ satisfies the first condition of separable extensions. The second condition is obvious.
\end{proof}

We establish the following result on homological dimensions of crossed products.

\begin{corollary}
Let $H$ be a subgroup of a finite group $G$. If $|G : H|$ is invertible in $A$, then $A _{\alpha} ^{\sigma} G$ and $A _{\alpha} ^{\sigma} H$ have the same global dimension, finitistic dimension, and strong global dimension.
\end{corollary}

\begin{proof}
Follows from Theorem 3.3 and Proposition 3.6.
\end{proof}

\section{Classify global dimensions}

As before, let $A$ be a left Noetherian associative ring, $G$ be a finite group, and $\mathfrak{I}$ be the left $A ^{\sigma} _{\alpha} G$ ideal generated by all elements in $\{ \sigma_x - 1 \mid 1_G \neq x \in G\}$. Recall that the trivial representation is $A ^{\sigma} _{\alpha} G / \mathfrak{I}$.

\begin{proposition}
If the trivial representation $A ^{\sigma} _{\alpha} G / \mathfrak{I}$ is projective, then an $A ^{\sigma} _{\alpha} G$-module $M$ is projective if and only if the restricted module $_AM$ is projective. In this situation, $A ^{\sigma} _{\alpha} H $ and $A$ have the same global dimension for every subgroup $H \leqslant G$.
\end{proposition}

\begin{proof}
Obviously, $_AM$ is projective if so is $M$. Conversely, assume that $_AM$ is projective. Note that $M$ is projective if and only if $\Hom _{A ^{\sigma} _{\alpha} G} (M, -)$ is exact. However, by Proposition 2.2, $\Hom _{A ^{\sigma} _{\alpha} G} (M, -) \cong \Hom_A (M, -)^G$ is the composite of two exact functors $\Hom_A (M, -)$ and $-^G \cong \Hom _{A ^{\sigma} _{\alpha} G} (A ^{\sigma} _{\alpha} G / \mathfrak{I}, -)$, and hence is exact. The first statement is verified.

Note that we always have $\gldim A ^{\sigma} _{\alpha} G \geqslant A$, so it suffices to show $\gldim A \geqslant \gldim A ^{\sigma} _{\alpha} G$ under the assumption. This is true if $\gldim A = \infty$, so we assume that $\gldim A = s < \infty$.

For an arbitrary $M \in A ^{\sigma} _{\alpha} G \module$, choosing a projective resolution $P^{\bullet} \in K^- (_RP)$ and applying the restriction functor $\downarrow ^G_1$, we get a projective resolution $_AP^{\bullet}$ for $_AM$. As $\pd_A M \leqslant s$, $_A K_s$ is a projective $A$-module, where $K_s$ is the $s$-th syzygy of $P^{\bullet}$. Therefore, $K_s$ is projective as an $A ^{\sigma} _{\alpha} G$-module. In other words, $\pd _{A ^{\sigma} _{\alpha} G} M \leqslant s$, so $\gldim A ^{\sigma} _{\alpha} G \leqslant s = \gldim A$. But for every subgroup $H \leqslant G$, we always have $\gldim A ^{\sigma} _{\alpha} G \geqslant \gldim A ^{\sigma} _{\alpha} H \geqslant \gldim A$. This forces $\gldim A ^{\sigma} _{\alpha} H  = \gldim A$.
\end{proof}

Unfortunately, the structure of $\mathfrak{I}$, and hence that of the trivial representation are hard to exploit for general crossed products. In contrast, the trivial representation of a skew group ring has a very explicit description. In this situation, $\mathfrak{I}$ consists of elements $\sum _{x \in G} a_x \sigma_x$ with $\sum _{x \in G} a_x = 0$, and the skew group ring $A ^{\sigma} G$ acts on $A$ by $(\sum _{x \in G} a_x \sigma_x) \cdot a = \sum _{x \in G} a_x \sigma_x (a)$. More details can be found in \cite{ARS}.

For every subgroup $H \leqslant G$, we define a trace map $\trace: A \to A^H$ by sending $a \in A$ to $\sum _{x \in H} \sigma_x (a)$. This is an $A^H$-linear map. Conclusions in the next proposition can be found in various literatures; see \cite{A, AY, ARS}.

\begin{proposition}
The following are equivalent for a skew group ring $A^{\sigma} G$:
\begin{enumerate}
\item the trivial $A ^{\sigma} G$-module $A$ is projective;
\item $M \in A ^{\sigma} G \module$ is projective if and only if the restricted module $_AM$ is projective;
\item the trace map $\trace: A \to A^G$ is surjective;
\item there is a certain element $a \in A$ such that $\trace (a) = 1$.
\end{enumerate}
When $A$ is commutative, they are equivalent to the following condition: $A ^{\sigma} G$ is a separable extension over $A$.
\end{proposition}

\begin{proof}
We just proved that (1) implies (2) for crossed products, so it is also true for skew group rings. It is obvious that (2) implies (1). The equivalence between (3) and (4) is straightforward since the trace map is $A^G$-linear. Equivalence between (1) and (4) is precisely Proposition 4.1 in page 87 of \cite{ARS} by observing that the proof there actually works for arbitrary left Noetherian rings. For commutative $R$, the equivalence of (4) and the last condition is implied by Proposition 2.1 in \cite{NVV}.
\end{proof}

This propositions tells us that the projective dimension of the trivial representation plays an important role in determining the homological dimensions of skew group rings. The following theorem classifies global dimensions and finitistic dimensions for crossed products.

\begin{theorem}
Let $A ^{\sigma} _{\alpha} G$ be a crossed product, where $A$ is left Noetherian and $G$ is a finite group.
\begin{enumerate}
\item For every $M \in A ^{\sigma} _{\alpha} G \module$, $\pd _{A ^{\sigma} _{\alpha} G} M$ is either infinity or equal to $\pd_A M$ (when $\pd_A M < \infty$).
\item Correspondingly, $\gldim A ^{\sigma} _{\alpha} G$ is either infinity or equal to $\gldim A$ (when $\gldim A < \infty$), and $\fdim A ^{\sigma} _{\alpha} G = \fdim A$.
\item If the trivial representation $A ^{\sigma} _{\alpha} G /\mathfrak{I}$ is projective when viewed as an $A$-module, then $\gldim A ^{\sigma} _{\alpha} G$ is finite if and only so is $\gldim A$ and $A ^{\sigma} _{\alpha} G /\mathfrak{I}$ is a projective $A ^{\sigma} _{\alpha} G$-module.
\end{enumerate}
\end{theorem}

\begin{proof}
Note that for crossed products, the induction functor $A ^{\sigma} _{\alpha} G \otimes _A -$ and the coinduction functor $\Hom _A (A ^{\sigma} _{\alpha} G, -)$ are naturally isomorphic. Therefore, for $M \in A ^{\sigma} _{\alpha} G \module$ and $N \in A\module$, we have the Frobenius reciprocity
\begin{equation*}
\Hom _A (M, N) \cong \Hom _{A ^{\sigma} _{\alpha} G} (M, A ^{\sigma} _{\alpha} G \otimes_A N)
\end{equation*}
and the Eckmann-Shapiro formula (see Corollary 2.8.4 in \cite{BD}) for $n \geqslant 1$:
\begin{equation*}
\Ext^n _A (M, N) \cong \Ext^n _{A ^{\sigma} _{\alpha} G} (M, A ^{\sigma} _{\alpha} G \otimes_A N).
\end{equation*}

To prove (1), we only need to consider the case that $\pd _{A ^{\sigma} _{\alpha} G} M < \infty$. Let $r = \pd _{A ^{\sigma} _{\alpha} G} M$ and $s = \pd_A M$. By Remark 3.4, $s \leqslant r$.

For an arbitrary $N \in A ^{\sigma} _{\alpha} G \module$, since the map $\varphi: A ^{\sigma} _{\alpha} G \otimes _A N \to N$ given by $\lambda \otimes v \mapsto \lambda v$ is a surjective $A ^{\sigma} _{\alpha} G$-module homomorphism, we have a short exact sequence $0 \to K \to A ^{\sigma} _{\alpha} G \otimes _A N \to N \to 0$. Applying $\Hom _{A ^{\sigma} _{\alpha} G} (M, -)$ to it, we get a long exact sequence:
\begin{align*}
& \ldots \to \Ext _{A ^{\sigma} _{\alpha} G} ^r (M, K) \to \Ext _{A ^{\sigma} _{\alpha} G} ^r (M, A ^{\sigma} _{\alpha} G \otimes _A N) \to\\
& \Ext _{A ^{\sigma} _{\alpha} G} ^r (M, N) \to \Ext _{A ^{\sigma} _{\alpha} G} ^{r+1} (M, K) \to \ldots
\end{align*}
But $\Ext _{A ^{\sigma} _{\alpha} G} ^{r+1} (M, K) = 0$ since $\pd _{A ^{\sigma} _{\alpha} G}  M = r$. If $\pd_A M = s < r$, then
\begin{equation*}
\Ext _{A ^{\sigma} _{\alpha} G} ^r (M, A ^{\sigma} _{\alpha} G \otimes N) \cong \Ext _A^r (M, N) = 0
\end{equation*}
by the Eckmann-Shapiro formula. Consequently, $\Ext _{A ^{\sigma} _{\alpha} G} ^r (M, N) = 0$. But $N \in A ^{\sigma} _{\alpha} G \module$ is arbitrary. Therefore, $\pd _{A ^{\sigma} _{\alpha} G} M < r$. This contradiction tells us that $s = r$, and (1) is established.

To classify global dimension of $A ^{\sigma} _{\alpha} G$, we still only need to consider the case that $\gldim A ^{\sigma} _{\alpha} G < \infty$. In this situation, $\pd _{A ^{\sigma} _{\alpha} G} M < \infty$ for every $M \in A ^{\sigma} _{\alpha} G \module$. Therefore, by (1), $\pd _{A ^{\sigma} _{\alpha} G} M = \pd_A M$. Consequently, $\gldim A ^{\sigma} _{\alpha} G \leqslant \gldim A$. But by Remark 3.4, $\gldim A ^{\sigma} _{\alpha} G \geqslant \gldim A$, and the equality follows.

Note that $\fdim A ^{\sigma} _{\alpha} G = \sup \{ \pd _{A ^{\sigma} _{\alpha} G} M \mid \pd _{A ^{\sigma} _{\alpha} G} M < \infty \}$. By (1), if $\pd _{A ^{\sigma} _{\alpha} G} M < \infty$, then $\pd _{A ^{\sigma} _{\alpha} G} M = \pd_A M \leqslant \fdim A$. Consequently, $\fdim A ^{\sigma} _{\alpha} G \leqslant \fdim A$. But by Remark 3.4, we have $\fdim A ^{\sigma} _{\alpha} G \geqslant \fdim A$. This forces $\fdim A ^{\sigma} _{\alpha} G = \fdim A$.

Now we turns to (3). If $A ^{\sigma} _{\alpha} G / \mathfrak{I}$ is projective as an $A ^{\sigma} _{\alpha} G$-module, by Proposition 4.1 $\gldim A ^{\sigma} _{\alpha} G = \gldim A$. Therefore, $\gldim A < \infty$ implies $\gldim A ^{\sigma} _{\alpha} G < \infty$. Conversely, if $\gldim A ^{\sigma} _{\alpha} G < \infty$, clearly $\gldim A < \infty$ and $\pd _{A ^{\sigma} _{\alpha} G} A ^{\sigma} _{\alpha} G / \mathfrak{I} < \infty$. By the first statement, $\pd _{A ^{\sigma} _{\alpha} G} A ^{\sigma} _{\alpha} G / \mathfrak{I} = \pd_A A ^{\sigma} _{\alpha} G / \mathfrak{I} = 0$.
\end{proof}

We have the following corollary for skew group rings:

\begin{corollary}
Let $A ^{\sigma} G$ be a skew group ring with $A$ left Noetherian and $G$ a finite group.
\begin{enumerate}
\item The skew group ring $A ^{\sigma} G$ has finite global dimension if and only if so does $A$ and the trivial representation is projective.
\item If the trivial representation is a projective $A ^{\sigma} G$-module, then $A ^{\sigma} G$ and $A$ have the same global dimension.
\item If $A$ is commutative and $A ^{\sigma} G$ has finite global dimension, then $A ^{\sigma} G$ is a separable extension over $A$. In particular, $A ^{\sigma} G$ and $A$ have the same global dimension, finitistic dimension and strong global dimension.
\end{enumerate}
\end{corollary}

\begin{proof}
The first statement follows from (3) of the above theorem since for skew group rings, the trivial representation $A$ is a free $A$-module. The second one follows from Proposition 4.1. Note that according to (1), the given condition in (3) implies that the trivial representation $A$ is a projective $A ^{\sigma} G$-module, so $A ^{\sigma} G$ is a separable extension over $A$ by the second part of Proposition 4.2.
\end{proof}

It is surprising in some sense to the author that $A ^{\sigma} _{\alpha} G$ and $A$ always have the same finitistic dimension. This makes the conjecture posted in \cite{L2} by the author trivial, which asks whether $\fdim A ^{\sigma} G < \infty$ whenever $\fdim A < \infty$.

We end this section by an example, which tells us that for noncommutative rings, separable extension in general is much stronger than the condition that the trivial representation is projective.

\begin{example} \normalfont
Let $A$ be the path algebra of the following quiver with relations $\beta \gamma = \gamma \beta = 0$ over an algebraically closed field $k$ with characteristic 2. Let $G$ be a cyclic group of order 2 generated by $g$, which permutes vertices $x$ and $y$, and arrows $\beta$ and $\gamma$. This action determines a skew group algebra $A ^{\sigma} G$.
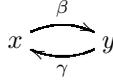
\begin{equation*}
\xymatrix{ x \ar@/^/[r] ^{\beta} & y \ar@/^/[l] ^{\gamma} }
\end{equation*}

It is easy to check that the center of $A$ is the one dimensional space spanned by $1_A = 1_x + 1_y$. However, since $g$ fixes every scalar, the trace map sends every element in the center of $A$ to 0. Therefore, by Proposition 2.1 in \cite{NVV}, $A ^{\sigma} G$ is not a separable extension over $A$. On the other hand, we can check that the trace map sends both $1_x$ and $1_y$ to the identity of $A ^{\sigma} G$. By Proposition 4.2 and Corollary 4.4, the trivial representation $A$ is a projective $A ^{\sigma} G$-module. Therefore, $A$ and $A ^{\sigma} G$ have the same global dimension (which is $\infty$) and finitistic dimension (which is 0).

Actually, since $1_x$ and $1_y$ are isomorphic idempotent in $A ^{\sigma} G$, the skew group algebra $A ^{\sigma} G$ in this example is actually Morita equivalent to $1_x A^{\sigma} G 1_x$. A direct computation shows that $1_x A^{\sigma} G 1_x \cong A^G \cong k[X] / (X^2)$.
\end{example}

\section{Crossed products of semiprimary algebras}

In this section let $A$ be a left Noetherian \textit{semiprimary} algebra over an algebraically closed field $k$ with characteristic $p \geqslant 0$. That is, the Jacobson radical $\mathfrak{R}$ of $A$ is nilpotent and $A / \mathfrak{R}$ is a finite dimensional $k$-algebra. By the main result of last section, $A _{\alpha} ^{\sigma} G$ has the same homological dimensions as $A _{\alpha} ^{\sigma} S$ for every Sylow $p$-subgroup $S \leqslant G$. Thus we mainly focus on $A _{\alpha} ^{\sigma} S$ in this section.

Take a complete set $E = \{ e_i \} _{i \in [n]}$ of primitive orthogonal idempotents in $A$. Then $_AA = \oplus _{i \in [n]} Ae_i$. Throughout this section we assume that there is a Sylow $p$-subgroup $S \leqslant G$ such that $E$ is an $S$-set; that is, $E$ is closed under the action of $S$.

Two elements $e, f \in E$ are said to be \textit{isomorphic} if $Ae \cong Af$ as $A$-modules. Note that $e$ and $f$ are isomorphic if and only if there are elements $u, v \in A$ such that $uv = e$ and $vu = f$. When identifying $e_i$ with $e_i1_S$, elements in $E$ are pairwise orthogonal idempotents in the crossed product $A _{\alpha} ^{\sigma} S$. It is obvious that isomorphic idempotents in $A$ are still isomorphic regarded as idempotents in $A ^{\sigma} _{\alpha} S$. Moreover, for every $e \in E$ and $x \in S$, $e$ and $\sigma_x (e)$ are isomorphic in $A _{\alpha} ^{\sigma} S$. We will show that $E$ is also a complete set of primitive orthogonal idempotents in $A _{\alpha} ^{\sigma} S$.

Clearly, for every $x \in S$, $\sigma_x$ maps $\mathfrak{R}$ onto $\mathfrak{R}$. In particular, $\mathfrak{R} S = \bigoplus _{x \in S} \mathfrak{R} \sigma_x$ is a two-sided ideal of $A _{\alpha} ^{\sigma} S$. Moreover, By Corollary 3.12 in page 86 of \cite{K}, $\mathfrak{R}$ is the intersection of the radical of $A _{\alpha} ^{\sigma} S$ and $A$. Therefore, $\mathfrak{R}$ is contained in the radical of $A _{\alpha} ^{\sigma} S$, so is $\mathfrak{R} S$. Let $\overline {A _{\alpha}^{\sigma} S}$ be the quotient algebra
\begin{equation*}
A _{\alpha}^{\sigma} S / \mathfrak{R} S = \bigoplus _{x \in S} A \sigma_x / \bigoplus _{x \in S}\mathfrak{R} \sigma_x \cong \bigoplus _{x \in S} (A/\mathfrak{R}) \sigma_x = \bar{A} _{\alpha}^{\sigma} S,
\end{equation*}
which is a crossed product as well. Since $\mathfrak{R} S$ is contained in the radical of $A _{\alpha}^{\sigma} S$, a complete set of primitive orthogonal idempotents in $\bar{A} _{\alpha}^{\sigma} S$ can be obtained from a complete set of primitive orthogonal idempotents in $A _{\alpha}^{\sigma} S$ by taking quotients. Conversely, given a complete set of primitive orthogonal idempotents in $\bar{A} _{\alpha}^{\sigma} S$, it can be lifted to a complete set of primitive orthogonal idempotents in $A _{\alpha}^{\sigma} S$.

With this observation, we have:

\begin{lemma}
The chosen set $E$ is a complete set of primitive orthogonal idempotents in $A _{\alpha}^{\sigma} S$.
\end{lemma}

\begin{proof}
By the above observation, we only need to show that $E$ is a complete set of primitive idempotents in $\bar{A} _{\alpha}^{\sigma} S$. Clearly, it is enough to prove that every idempotent in $E$ is primitive in $\bar{A} _{\alpha} ^{\sigma} S$. We deduce it by showing that the algebra $e_i (\bar{A} _{\alpha} ^{\sigma} S) e_i$ is a finite dimensional local algebra for every $i \in [n]$; here and later by convention we identity idempotents with their images in the quotient algebra. Since $e_i$ is a primitive idempotent in $A$, it is still a primitive idempotent in $\bar{A} = A / \mathfrak{R}$.

Note that
\begin{equation*}
e_i (\bar{A} _{\alpha} ^{\sigma} S) e_i = \bigoplus _{x \in S} e_i \bar{A} \sigma_x (e_i) \sigma_x.
\end{equation*}
Let $H = \{ x \in S \mid \bar{A} \sigma_x(e_i) \cong \bar{A}e_i \}$. Because $\bar{A}$ is semisimple, $H = \{ x \in S \mid e_i \bar{A} \sigma_x (e_i) \neq 0 \}$. This is a subgroup of $S$. Indeed, for $x, y \in H$, suppose that $e \bar{A} \sigma_x (e) \neq 0$ and $e \bar{A} \sigma_y (e) \neq 0$. That is, $\bar{A} e \cong \bar{A} \sigma_x(e) \cong \bar{A} \sigma_y (e)$. Applying $\sigma_x$ to $e \bar{A} \sigma_x(e)$ we get $\sigma_x (e) A \alpha (x, x) \sigma_{x^2} (e) \alpha (x, x)^{-1} \neq 0$, so $\sigma_x(e) \bar{A} \sigma_{x^2} (e) \neq 0$. Therefore, $\bar{A} e \cong \bar{A} \sigma_x (e) \cong \bar{A} \sigma_{x^2} (e)$. Repeating this process, we have $\bar{A} \sigma_{x^{-1}} (e) \cong \bar{A} e \cong \bar{A} \sigma_y (e)$, so $\sigma _{x^{-1}} (e) \bar{A} \sigma_y (e) \neq 0$. Applying $\sigma_x$ one more time we deduce that $e \bar{A} \sigma_{xy} (e) \neq 0$. That is, $xy \in H$, so $H$ is a group.

We observe that
\begin{equation*}
e_i (\bar{A} _{\alpha} ^{\sigma} S) e_i = \bigoplus _{x \in H} e_i \bar{A} \sigma_x (e_i) \sigma_x.
\end{equation*}
is a strongly $H$-graded algebra (\cite{K}). Moreover, $e_i \bar{A} e_i \cong k$ since $k$ is algebraically closed. By Proposition 1.11 in page 71 of \cite{K}, it is a crossed product of $k$ with $H$, and hence a twisted group algebra. It is well known that this is a local algebra.\footnote{Actually, by changing basis, this twisted group algebra is isomorphic to an ordinary group algebra. For details, see the proof of the theorem in Section 1 of \cite{C}. This can also be deduced from the fact that $H^2 (S, k) = 1$ for a $p$-group $S$; see for example Proposition 6.1 in page 42 of \cite{K}.} The conclusion follows.
\end{proof}

The following proposition motivates us to consider free action of $S$ on $E$.

\begin{proposition}
Let $A ^{\sigma} _{\alpha} S$ and $E$ be as before. If $A$ is finite dimensional and the action of $S$ on $E$ is not free, then $\gldim A ^{\sigma} _{\alpha} S = \infty$.
\end{proposition}

\begin{proof}
Since the action of $S$ on $E$ is not free, we can take some $e \in E$ and some $1 \neq x \in S$ such that $\sigma_x (e) = e$. Let $H = \langle x \rangle$, which is a nontrivial cyclic $p$-group. By Remark 3.4, it suffices to show that $\gldim A ^{\sigma} _{\alpha} H = \infty$.

Consider the quotient algebra $A ^{\sigma} _{\alpha} H / \mathfrak{R} H \cong \bar{A} ^{\sigma} _{\alpha} H$. Then $\bar{A} ^{\sigma} _{\alpha} H e$ is a projective $\bar{A} ^{\sigma} _{\alpha} H$-module, and clearly an $A ^{\sigma} _{\alpha} H$-module. Since elements in $H$ fix $Ae$, they fix $\bar{A}e$ as well, and we have
\begin{equation*}
\bar{A} _{\alpha} ^{\sigma} H e = \bigoplus _{x \in H} \bar{A} \sigma_x e = \bigoplus _{x \in H} \bar{A} \sigma_x (e) \sigma_x = \bigoplus _{x \in H} \sigma_x (\bar{A}e) \sigma_x = \bigoplus _{x \in H} \bar{A}e \sigma_x = \bar{A} eH,
\end{equation*}
and hence
\begin{equation*}
\Hom _{A ^{\sigma} _{\alpha} H} (A ^{\sigma} _{\alpha} He, \bar{A} ^{\sigma} _{\alpha} H e) \cong e (\bar{A} ^{\sigma} _{\alpha} H) e = (e \bar{A} e) H \cong k_{\alpha} H
\end{equation*}
is a twisted group algebra with nontrivial $H$. Since it is not a division ring, we conclude that $\bar{A} ^{\sigma} _{\alpha} H e$ is not a simple $A ^{\sigma} _{\alpha} H$-module.

Take $f \in E$ such that $A ^{\sigma} _{\alpha} Hf \ncong A ^{\sigma} _{\alpha} He$. Because isomorphic idempotents in $A$ viewed as idempotents in $A ^{\sigma} _{\alpha} H$ are still isomorphic, we deduce that $e$ and $f$ are not isomorphic in $A$, so $fAe \in \mathfrak{R}$, and hence $f \bar{A} e = 0$. Therefore,
\begin{equation*}
\Hom _{A ^{\sigma} _{\alpha} H} (A ^{\sigma} _{\alpha} Hf, \bar{A} ^{\sigma} _{\alpha} H e) \cong f (\bar{A} ^{\sigma} _{\alpha} H) e = (f \bar{A} e) H = 0,
\end{equation*}
which implies that all composition factors of $\bar{A} _{\alpha} ^{\sigma} H e$ are isomorphic to $S_e$, the simple $A _{\alpha} ^{\sigma} H$-module corresponding to the primitive idempotent $e$. Consequently, $(\bar{A} _{\alpha} ^{\sigma} H) e$ is a non-simple $A _{\alpha} ^{\sigma} H$-module with only composition factors isomorphic to $S_e$.

The short exact sequences $0 \to M \to \bar{A} _{\alpha} ^{\sigma} H e \to S_e \to 0$ and $0 \to \mathfrak{R} He \to A _{\alpha} ^{\sigma} He \to \bar{A} _{\alpha} ^{\sigma} H e \to 0$ give rise to the following diagram with exact rows and columns:
\begin{equation*}
\xymatrix{ & 0 \ar[r] & \mathfrak{R} He \ar[r] \ar[d] & \Omega S_e \ar[r] \ar[d] & M \ar[r] & 0\\
 & & A _{\alpha} ^{\sigma} H e \ar@{=}[r] \ar[d] & A _{\alpha} ^{\sigma} He \ar[d] \\
0 \ar[r] & M \ar[r] & \bar{A} _{\alpha} ^{\sigma} He \ar[r] & S_e \ar[r] & 0}
\end{equation*}
where $\Omega$ is the Heller operator. Note that $M \neq 0$ and has only composition factors isomorphic to $S_e$.

Applying $\Hom _{A _{\alpha} ^{\sigma} H} (-, S_e)$ to the last column we conclude that $\Ext_{A _{\alpha} ^{\sigma} H} ^1 (S_e, S_e) \cong \Hom _{A _{\alpha} ^{\sigma} H} (\Omega S_e, S_e)$. Applying the same functor to the top row we get an inclusion $\Hom _{A _{\alpha} ^{\sigma} H} (M, S_e) \to \Hom_ {A _{\alpha} ^{\sigma} H} (\Omega S_e, S_e)$. Since $\Hom _{A _{\alpha} ^{\sigma} H} (M, S_e) \neq 0$, we deduce that $\Ext _{A _{\alpha} ^{\sigma} H} ^1 (S_e, S_e) \neq 0$. By the strong no loop conjecture proved in \cite{ILP}, the projective dimension $\pd _{A _{\alpha} ^{\sigma} H} S_e = \infty$. Therefore, $\gldim A _{\alpha} ^{\sigma} H = \infty$.
The conclusion follows.
\end{proof}

Here the proof relies on the strong no loop conjecture, which requires $A$ to be an artinian $k$-algebra. We wonder if there is an alternate proof for arbitrary left Noetherian semiprimary $k$-algebras.

We describe two corollaries.

\begin{corollary}
Let $A$ be a finite dimensional algebra. Then a twisted group algebra $A_{\alpha} G$ has finite global dimension if and only if $\gldim A$ is finite and the order of $G$ is invertible. Moreover, in this situation we have $\gldim A = \gldim A _{\alpha} G$.
\end{corollary}

\begin{proof}
We always have $\gldim A_{\alpha} G = \gldim A_{\alpha} S \geqslant \gldim A$. Assume that $\gldim A_{\alpha} G < \infty$, then $\gldim A < \infty$. Moreover, since $G$ acts trivially on $A$, every chosen complete set of primitive orthogonal idempotents $E$ is closed under the action of $G$, and hence closed under the trivial action of $S$. But by the previous proposition, $S$ must acts freely on $E$. This forces $S = 1$. That is, $|G|$ is invertible. Conversely, if $|G|$ is invertible, then $\gldim A_{\alpha} G = \gldim A _{\alpha} S$. But $A _{\alpha} S = A$. The conclusion follows.
\end{proof}

Let $K$ be the kernel of $\sigma: S \to \Aut_k (A)$, which is a subgroup of $S$. The action of $S$ on $A$ is said to be \textit{faithful} if $K$ is the trivial group.

\begin{corollary}
Let $A$ be a finite dimensional algebra. If the action of $S$ on $A$ is not faithful, then $\gldim A _{\alpha} ^{\sigma} G = \infty$.
\end{corollary}

\begin{proof}
Since the action of $S$ on $A$ is not faithful, we can find a nontrivial subgroup $H \leqslant S$ such that every element in $H$ acts on $A$ trivially. In particular, for a chosen complete set of primitive orthogonal idempotents $E$, it is closed under the trivial action of $H$. By Proposition 5.2, $\gldim A _{\alpha} ^{\sigma} H = \infty$. But we always have $\gldim A _{\alpha} ^{\sigma} G \geqslant \gldim A _{\alpha} ^{\sigma} H$. The conclusion follows.
\end{proof}

However, even if $\gldim A < \infty$ and $S$ acts faithfully on $E$, the global dimension of $A _{\alpha} ^{\sigma} G$ might not be finite.

\begin{example} \normalfont
Let $A$ be the path algebra of the following quiver over an algebraically closed field of characteristic 2, and let $G = S = \langle x \rangle$ be a cyclic group of order 2. The action of $S$ on $A$ is determined by $x (1) = 3$ and $x(2) = 2$. This action is faithful.
\begin{equation*}
\xymatrix{ 1 \ar[r] & 2 & 3 \ar[l]}
\end{equation*}
However, the skew group algebra $A ^{\sigma} S$ is Morita equivalent to the path algebra of the following quiver with relations $\delta^2 = 0$, which has infinite global dimension. This is because the action of $S$ on the chosen set $E = \{e_1, e_2, e_3\}$ of primitive idempotents is not free.
\begin{equation*}
\xymatrix{ 1 \ar[r] \ar@(ul,dl)[]|{\delta}& 2}
\end{equation*}
\end{example}

Let $s$ denote the number of $S$-orbits in $E$ and take a representative $e_i$ from each $S$-orbit. Without loss of generality, we can assume that $e_1, e_2, \ldots, e_s$ give a chosen set of representatives from distinct orbits. Let $\epsilon = e_1 + e_2 + \ldots + e_s$. A special case of the following result is described in Proposition 1.6 in page 67 of \cite{K}.

\begin{proposition}
Let $A _{\alpha} ^{\sigma} S$, $E$, and $\epsilon$ be as above and suppose that the action of $S$ on $E$ is free. Then $A _{\alpha} ^{\sigma} S$ is an $|S| \times |S|$ matrix algebra over $\epsilon A ^{\sigma} _{\alpha} S \epsilon$.
\end{proposition}

\begin{proof}
Note that for every $e \in E$ and $x \in S$, $A _{\alpha} ^{\sigma} S e$ and $A _{\alpha} ^{\sigma} S \sigma_x (e)$ are isomorphic. Indeed, let $\mu = \sigma_x(e) \sigma_x$ and $\nu = \alpha (x^-, x)^{-1} \sigma_{x^-}$ where $x^- = x^{-1}$. Using $\sigma_x (\alpha (x^-, x)) = \alpha (x, x^-)$, we get
\begin{equation*}
\mu \ast \nu = (\sigma_x(e) \sigma_x) \ast (\alpha (x^-, x)^{-1} \sigma_{x^-}) = \sigma_x(e) \sigma_x (\alpha (x^-, x)^{-1}) \alpha (x, x^-) = \sigma_x (e);
\end{equation*}
and
\begin{equation*}
\nu \ast \mu = (\alpha (x^-, x)^{-1} \sigma_{x^-}) \ast (\sigma_x (e) \sigma_x) = (\alpha (x^-, x)^{-1} \sigma_{x^-}) \ast \sigma_x \ast e = e.
\end{equation*}
Consequently, $A _{\alpha} ^{\sigma} S \epsilon \cong A _{\alpha} ^{\sigma} S \sigma_x (\epsilon)$. Since the action of $S$ on $E$ is free, and by our definition of $\epsilon$, we have
\begin{equation*}
_{A _{\alpha} ^{\sigma} S} A _{\alpha} ^{\sigma} S = \bigoplus _{x \in S} A _{\alpha} ^{\sigma} S \sigma_x (\epsilon) \cong (A _{\alpha} ^{\sigma} S \epsilon) ^{|S|}.
\end{equation*}
Therefore, $A _{\alpha} ^{\sigma} S$ is a matrix algebra over $\epsilon A _{\alpha} ^{\sigma} S \epsilon = \End _{A _{\alpha} ^{\sigma} S} (A _{\alpha} ^{\sigma} S \epsilon) ^{\textnormal{op}}$, and is Morita equivalent to $\epsilon A ^{\sigma} _{\alpha} S \epsilon$.
\end{proof}

In the situation that $\epsilon$ is a \textit{central idempotent} in $A$, that is, $\epsilon a = a \epsilon$ for every $a \in A$, we get a very simple case. Note that $\epsilon$ might not be in the center of $A _{\alpha} ^{\sigma} S$ as $\sigma_x \epsilon = \sigma_x (\epsilon) \sigma_x \neq \epsilon \sigma_x$. Since $\sigma_x$ acts on $A$ as an algebra automorphism, $\sigma_x (\epsilon)$ is a central idempotent in $A$ as well for every $x \in S$. Then
\begin{equation*}
A = \bigoplus _{x \in S} A \sigma_x (\epsilon)  = \bigoplus_{x \in S} \sigma_x (\epsilon) A \sigma_x (\epsilon)
\end{equation*}
is actually a direct sum of isomorphic algebras. Moreover, we have:
\begin{equation*}
\epsilon A _{\alpha} ^{\sigma} S \epsilon = \bigoplus _{x \in S} \epsilon A \sigma_x \epsilon = \bigoplus _{x \in S} \epsilon A \sigma_x (\epsilon) \sigma_x = \epsilon A \epsilon
\end{equation*}
since $\epsilon$ is in the center of $A$, and $\epsilon$ and $\sigma_x (\epsilon)$ are orthogonal for $1_S \neq x \in S$. Therefore, $A _{\alpha} ^{\sigma} S$ is actually an $|S| \times |S|$ matrix algebra over $\epsilon A \epsilon$, while $A$ is the subalgebra of all diagonal matrices. In particular, $A _{\alpha} ^{\sigma} S$, $A$, and $A^S$ all have the same homological dimensions, where $A^S$ is the fixed algebra; that is, $A^S = \{ a \in A \mid \sigma_x (a) = a \, \forall x \in S \}$.

For general crossed products, the structure of $A^S$ is hard to explore. However, if $\alpha (x, y)$ is contained in the center of $A$ for all $x, y \in S$ and the action of $S$ is free on $E$, we have an explicit description of $A^S$.

\begin{proposition}
If $S$ acts on $E$ freely and $\alpha (x, y)$ is contained in the center of $A$ for all $x, y \in S$, then the fixed algebra $A^S = \{ \sum _{x \in S} \sigma_x (a) \mid a \in A \}$. Furthermore, $A \sigma_x (\epsilon) \cong A^S$ (resp, $\sigma_x (\epsilon) A \cong A^S$) as left (resp., right) $A^S$-modules for every $x \in S$.
\end{proposition}

\begin{proof}
Since $\alpha (x, y)$ lies in the center of $A$, $A$ is a $kS$-module, so $\sum _{x \in S} \sigma_x (a) \in A^S$ for every $a \in A$. To show the other inclusion, we take an element $a \in A^S$. Then $a = \sum _{x \in S} a \sigma_x (\epsilon) = \sum _{x \in S} \sigma_x (a \epsilon)$. The first statement is proved.

Since $A^S$ is a subalgebra of $A$, $A \sigma_x (\epsilon)$ is a left $A^S$-module for every $x \in S$. Define $\varphi: A^S \to A \sigma_x (\epsilon)$ by letting $\varphi(a) = a \sigma_x (\epsilon)$, and $\psi: A \sigma_x (\epsilon) \to A^S$ by sending $a \sigma_x (\epsilon)$ to $\sum_{y \in S} \sigma_y (a \sigma_x (\epsilon))$, which is contained in $A^S$. We check that both $\varphi$ and $\psi$ are $A^S$-module homomorphisms. Moreover, for $a \in A^S$,
\begin{equation*}
\psi (\varphi (a)) = \psi (a \sigma_x (\epsilon)) = \sum _{y \in S} \sigma_y (a \sigma_x (\epsilon)) = \sum _{y \in S} a \sigma_{yx} (\epsilon) = a;
\end{equation*}
and for $b \sigma_x (\epsilon) \in A \sigma_x (\epsilon)$,
\begin{equation*}
\varphi (\psi (b \sigma_x (\epsilon))) = \varphi (\sum_{y \in S} \sigma_y (b \sigma_x (\epsilon))) = \sum _{y \in S} \sigma_y (b) \sigma_{yx} (\epsilon) \sigma_x (\epsilon) = b \sigma_x (\epsilon).
\end{equation*}
Therefore, $\varphi$ and $\psi$ are inverse to each other, so $A\epsilon \cong A^S$ as left $A^S$-modules. The conclusion for right modules can be proved similarly.
\end{proof}

From this proposition we immediately deduce that $A$ is both a left and a right free $A^S$-module whenever the action of $S$ on $E$ is free and $\alpha (x,y)$ is contained in the center of $A$ for $x, y \in S$. However, usually $A$ is not a free $A^S$-bimodule; see Example 3.6 in \cite{L1}.

Recall that a ring $R$ is a \textit{domain} if it has no nonzero zero factors. Denote by $C(A)$ the center of $A$, and by $U(A)$ the multiplicative group of invertible elements in $A$. The following technical lemma will be used for normalization of crossed products.

\begin{lemma}
Let $D \subseteq C(A)$ be a domain.
\begin{enumerate}
\item For every $a \in D$ and $n \in \mathbb{N}$ which is a power of $p$, the polynomial $X^n - a \in D[X]$ has at most one root in $D$.
\item If $a \in A^S \cap U(A)$ and $\lambda \in D$ is a root of the above polynomial, then $\lambda \in A^S \cap U(A)$ as well.
\end{enumerate}
\end{lemma}

\begin{proof}
If $\lambda_1, \lambda_2 \in D$ are two roots, then we have $\lambda_1^n - \lambda_2^n = 0$. But since $D$ is commutative and $n$ is a power of $p$, we have $(\lambda_1 - \lambda_2)^n = 0$. As $D$ is assumed to be a domain, this happens if and only if $\lambda_1 = \lambda_2$, so (1) is true.

If $a$ is invertible, then $aa^{-1} = 1$. But $a = \lambda^n$, so $\lambda (\lambda^{n-1} a^{-1}) = (\lambda^{n-1} a^{-1}) \lambda = 1$. That is, $\lambda \in U(A)$. For every $x \in S$, we have $(\sigma_x (\lambda))^n = \sigma_x (\lambda^n) = \sigma_x (a) = a$. That is, $\sigma_x (\lambda)$ is also a root of the polynomial $X^n - a$. But the root is unique, so $\sigma_x (\lambda) = \lambda$; i.e., $\lambda \in A^S$.
\end{proof}

We introduce some notation. For $x \in S$, define $h_x = \prod _{y \in S} \alpha (x, y)$. This is well defined since $\alpha (x, y)$ is contained in $C(A)$ for all $x, y \in S$. Because $\alpha (x, y) \alpha(xy, z) = \sigma_x (\alpha (y, z)) \alpha (x, yz)$, letting $z$ range over all elements in $S$ and taking the product, we get
\begin{equation}
\alpha (x, y)^{|S|} h_{xy} = \sigma_x (h_y) h_x.
\end{equation}

\begin{lemma}
Suppose that the following conditions hold:
\begin{enumerate}
\item There is a domain $D \subseteq C(A) \cap A^S$ containing all $\alpha (x, y)$ for $x, y \in S$.
\item The $|S|$-th root of $h_x$ exists in $D$ for every $x \in S$.
\end{enumerate}
Then $A _{\alpha} ^{\sigma} S$ is equivalent to a skew group algebra $A ^{\sigma'} S$.
\end{lemma}

\begin{proof}
For each $x \in S$, let $u_x \in D$ be the $|S|$-th root of $h_x$, which is unique and is contained in $U(A) \cap A^S$ as well by the previous lemma. Now we define another parameter set $(\alpha', \sigma')$ by letting $\sigma_x' = u_x^{-1} \sigma_x$. Then $\alpha'(x, y) = u_x^{-1} u_y^{-1} \alpha(x, y) u_{xy}$. Taking the $|S|$-th power we get
\begin{equation*}
\alpha'(x, y) ^{|S|} = h_x^{-1} h_y^{-1} h_{xy} \alpha (x, y)^{|S|} = 1
\end{equation*}
by (5.1). Note that $\alpha' (x, y)$ is also contained in $D$. Therefore, $\alpha'(x, y) = 1$.

We have proved that the parameter set $(\alpha, \sigma)$ is equivalent to $(\alpha', \sigma')$ with $\alpha' (x, y) = 1$. Therefore, by Proposition 2.1, $A _{\alpha} ^{\sigma} S$ is equivalent to a skew group algebra $A ^{\sigma'} S$.
\end{proof}

Note that $\sigma_x' = u_x^{-1} \sigma_x$. Therefore, $a \in A$ is fixed by $\sigma_x$ if and only it is fixed by $\sigma_x'$. In other words, the fixed algebras for these two equivalent crossed products coincide, so can be denoted by $A^S$ again.

Now we prove Theorem 1.3, generalizing a main result in \cite{L1, L2}.

\begin{theorem}
Suppose that conditions (1) and (2) in the previous lemma hold. Then:
\begin{enumerate}
\item If $A$ is a finite dimensional algebra, then $\gldim A _{\alpha} ^{\sigma} G < \infty$ if and only if $\gldim A < \infty$ and $S$ acts on $E$ freely.
\item If $S$ acts freely on $E$, then $A _{\alpha} ^{\sigma} G$ and $A$ have the same global dimension. Moreover, if $A$ as an $A^S$-bimodule has a summand $A^S$, then $A _{\alpha} ^{\sigma} G$ and $A$ have the same strong global dimension.
\end{enumerate}
\end{theorem}

\begin{proof}
By Theorem 1.1, we can assume that $G = S$. By the previous lemma, $A _{\alpha} ^{\sigma} S$ is isomorphic to the skew group algebra $A ^{\sigma'} S$. In particular, they have the same homological dimensions and the same fixed algebras. Then the conclusion follows from Theorem 1.1 in \cite{L1} and Theorem 1.1 in \cite{L2}.\footnote{Although in these two papers we mainly deal with finite dimensional algebras, the conclusion that $A$ and $A^{\alpha'} G$ share the same homological dimensions when $S$ acts freely on $E$ still holds for semiprimary Noetherian $k$-algebras.}
\end{proof}

\end{document}